\documentclass[12pt, reqno]{amsart}
\usepackage{amsmath, amsthm, amscd, amsfonts, amssymb, graphicx, color}
\usepackage[bookmarksnumbered, colorlinks, plainpages]{hyperref}

\newtheorem{theorem}{Theorem}[section]
\newtheorem{lemma}[theorem]{Lemma}

\newtheorem{corollary}[theorem]{Corollary}
\theoremstyle{definition}
\newtheorem{definition}[theorem]{Definition}
\newtheorem{example}[theorem]{Example}

\theoremstyle{remark}

\numberwithin{equation}{section}

\begin{document}

\title[Skew generalized power series rings with the McCoy property]{Skew generalized power series rings with \\ the McCoy property}

\author[P. Danchev]{P. Danchev}
\author[M. Zahiri]{M. Zahiri$^*$}
\author[S. Zahiri]{S. Zahiri}

\address{Peter Danchev, Institute of Mathematics and Informatics, Bulgarian Academy of Sciences, 1113 Sofia, Bulgaria}
\email{\textcolor[rgb]{0.00,0.00,0.84}{pvdanchev@yahoo.com; danchev@math.bas.bg}}
\address{Masoome Zahiri, Department of  Mathematics, Faculty of Sciences, Higher Education center of Eghlid, Eghlid, Iran}
\email{\textcolor[rgb]{0.00,0.00,0.84}{m.zahiri86@gmail.com}}
\address{Saeede Zahiri, Department of  Mathematics, Faculty of Sciences, Higher Education center of Eghlid, Eghlid, Iran}
\email{\textcolor[rgb]{0.00,0.00,0.84}{saeede.zahiri@yahoo.com}}

\thanks{*Corresponding author}

\keywords{Skew generalized power series ring; $(S,\omega)$-McCoy ring; strictly ordered monoid; unique product monoid; abelian ring; semi-regular ring.\\
\indent{2020 Mathematics Subject Classification}: Primary 16D15; Secondary 16D40, 16D70.}

\begin{abstract} Let $R$ be a ring, $(S,\preceq)$ a strictly totally ordered monoid and suppose also
$\omega:S\rightarrow \text{End}(R)$ is a monoid homomorphism. A skew generalized power series ring $R[[S,\omega,\preceq]]$ consists of all functions from a monoid $S$ to a coefficient ring $R$ whose support contains neither infinite descending chains nor infinite anti-chains, equipped with point-wise addition and with multiplication given by convolution twisted by an action $\omega$ of the monoid $S$ on the ring $R$.

Special cases of the skew generalized power series ring construction are the skew polynomial rings, skew Laurent polynomial rings, skew power series rings, skew Laurent series rings, skew monoid rings, skew group rings, skew
Malcev-Neumann series rings and generalized power series rings as well as the untwisted versions of all of these objects.

In the present article, we study the so-termed $(S,\omega)$-McCoy condition on $R$, that is a generalization of the standard McCoy condition from polynomials to skew generalized power series, thus generalizing some of the existing results in the literature relevant to the subject.
\end{abstract}

\maketitle


\section{Introduction and definitions}

Throughout this paper, all monoids and rings are assume to have identity. We use the notations $J(R)$ and $Ni\ell(R)$ to denote the Jacobson radical of a ring $R$ and the set of all nilpotent elements in $R$, respectively.\par

Let $(S,\leq)$ be an ordered set. Then, $(S,\leq)$ is said to be {\it artinian} if every strictly decreasing sequence of elements of $S$ is finite, and $(S,\leq)$ is said to be {\it narrow} if every subset of pair-wise order-incomparable elements of $S$ is finite. Thus, $(S,\leq)$ is artinian and narrow if, and only if, each non-empty subset of $S$ has at least one but only a finite number of minimal elements. Clearly, the union of a finite family of artinian and narrow subsets of an ordered set as well as any subset of an artinian and narrow set are again artinian
and narrow.\par

A monoid $S$ (everywhere written multiplicatively) equipped with an order $\leq$ is called an {\it ordered monoid} if, for any $s_{1},s_{2},t\in S$, the inequality $s_{1}\leq s_{2}$ implies $s_{1}t\leq s_{2}t$ and $ts_{1}\leq ts_{2}$. Moreover, if the inequality $s_{1} < s_{2}$ implies $s_{1}t < s_{2}t$ and $ts_{1} < ts_{2}$, then $(S,\leq)$ is called  {\it strictly ordered}.

Let $R$ be a ring, $(S, \leq)$ a strictly ordered monoid, and suppose $\omega: S\rightarrow \text{End}(R)$ is a monoid homomorphism. For $s \in S$, we put $\omega_{s}$ to denote the image of $s$ under $\omega$, that is, $\omega_{s}=\omega (s)$.

Likewise, let $A$ be the set of all functions $f: S\rightarrow R$ such that the set $$Supp(f)= \{s\in  S : f(s)\neq 0 \}$$ is artinian and narrow. Then, for any $s\in S$ and $f,g \in A $, the set
$$X_{s}( f, g)=\{(x, y) \in Supp(f) \times Supp(g):  s=xy\}$$
is finite. Thus, one can define the product $f g: S\rightarrow R$ of $f, g \in A$ as follows:
$$fg(s)=\sum_{(u,v)\in X_{s}(f,g)}f(u)\omega_{u} (g(v)),$$
(for our convention, the sum over the empty set is $0$). So, with point-wise addition and multiplication as defined above, $A$ becomes a ring called the {\it ring of skew generalized power series} with coefficients in $R$ and exponents in $S$ (one can think of a map $f:S\rightarrow R$ as a formal series $\sum_{s\in S}r_ss,$ where $r_s = f(s)\in R$) and denoted either by $R[[S,\omega,\leq]]$ (or just by $R[[S,\omega ]]$ when there is no ambiguity concerning the order $\leq$) (see, for more details, \cite{MMZ} and \cite{MZ1}).

Historically, the skew generalized power series construction embraces a wide range of classical ring-theoretic
extensions, including skew polynomial rings, skew power series rings, skew Laurent polynomial rings, skew group rings, Malcev-Neumann Laurent series rings, and of course the untwisted versions of all of these.\par

Hereafter, we will use the symbol $1$ to denote the identity elements of the monoid $S$, the ring $R$, and the ring $R[[S,\omega,\leq]]$, as well as the trivial monoid homomorphism $1: S\rightarrow \text{End}(R)$ that sends every element of $S$ to the identity endomorphism.

For each $r\in R$ and $s\in S$, let us define $c_{r},e_{s}\in R[[S,\omega,\leq]]$ by
$$c_{r}\left(x\right) =\left\{
\begin{array}{cc}
r & \text{if}~\; x=1 \\
0 & \text{if}~\; x\in S\backslash \left\{ 1\right\} ,%
\end{array}%
\right. \ e_{s}\left(x\right) =\left\{
\begin{array}{cc}
1 & \text{if}~\; x=s \\
0 &\text{if}~\; x\in S\backslash \left\{ s\right\} .%
\end{array}%
\right.$$

It is clear that $r\rightarrow c_{r}$ is a ring embedding of $R$ into $R[[S,\omega,\leq]]$, and $s\rightarrow e_{s}$ is a monoid embedding of $S$ into the multiplicative monoid of the ring $R[[S,\omega,\leq]]$, such that $e_{s}c_{r}= c_{\omega _{s}(r)}e_{s}$. Moreover, for any non-empty subset $X$ of $R$, we have
$$X[[S,\omega,\leq]]=\left\{f\in R[[S,\omega,\leq]]:~~~ f(s)\in X\cup\{0\}~~~~~ \text{for every}~~~~~ s\in S  \right\},$$

\noindent and, for each non-empty subset $Y$ of $R[[S,\omega,\leq]]$, we put
$$C_{Y}=\{g(t):~~ g\in Y ~~,~~ t\in S\}.$$
\par In \cite{PPN}, Nielsen introduced the McCoy rings. He call a ring $R$ {\it right McCoy} if the equation $f(x)g(x)=0$ forces $f(x)c=0$ for some non-zero $c\in R$, where $f(x),g(x)$ are non-zero polynomials in $R[x]$. {\it Left McCoy} rings are defined similarly. Generally, a ring $R$ is called {\it McCoy} if it is simultaneously left and right McCoy. The interested readers can be referred to \cite{mohamadi, mohamadi1, zah} for a more detailed information on this topic. \par

However, taking into account \cite{DEF}, the McCoy's theorem expectably fails in general for the case of formal power series ring $R[[x]]$ over a commutative ring $R$. In fact, Fields proved in \cite[Theorem 5]{DEF} that, if $R$ is a commutative Noetherian ring with identity in which $(0)=Q_{1}\cap Q_{2}\cap \cdots \cap Q_{n}$ is a shortest primary representation of $(0)$, then $f(x)=\sum_{i=0}^{\infty}a_ix^i\in R[[x]]$ is a zero-divisor in $R[[x]]$ if, and only if, there is a non-zero element $r\in R$ which satisfies $rf(x)=0$. He also provided an example showing that the condition ``$R$ is Noetherian'' is not redundant (cf. \cite[Example 3]{DEF}).\par

In the current work, we study the $(S,\omega)$-McCoy condition on $R$, which is a non-trivial generalization of the classical McCoy condition from polynomials to skew generalized power series. To summarize our achievements, we show that if $R$ is an $S$-compatible abelian semi-regular ring with $J(R)$ nilpotent, then $R$ is $(S,\omega)$-McCoy (see Theorem~\ref{main}). This answers in the affirmative a question posed in \cite{zahiri}.

\section{McCoy rings of skew generalized power series rings}

As usual, a ring $R$ is called {\it semi-regular} if $R/J(R)$ is regular and all idempotents can be lifted modulo $J(R)$. Every semi-regular ring is an exchange ring. Note that the class of semi-regular rings are quite large; indeed, regular rings and semi-perfect rings are known to be semi-regular. Furthermore, it is also principally known that the endomorphism rings of injective modules are semi-regular. For a more deep study of semi-regular
rings and related topics, we refer the interested readers to \cite{Nic1} and \cite{Nic2}.\par

In addition, Yang et al. \cite{YSL} continued to work in this area, introducing the concept of a power series-wise McCoy ring: a ring $R$ is called {\it right power series-wise McCoy} if, for any $f(x),g(x)\in R[[x]]$, there exists a non-zero annihilator $c\in R$ with $f(x)c=0$, satisfying also the equality $f(x)g(x)=0$. \par

\medskip

We now come to our basic notion, mainly motivating our investigation.

\begin{definition}\label{skew McCoy} A ring $R$ is called \textit{right $(S,\omega)$-McCoy} if, whenever $fg=0$ for non-zero elements $f,g\in R[[S,\omega,\leq]]$, then there exists a non-zero element $c\in R$ such that $fc=0$ or, equivalently, $f(s).\omega_{s}(c)=0$ for all $s\in S$. \textit{Left $(S,\omega)$-McCoy} rings are defined by analogy, and a ring $R$ is generally called {\it $(S,\omega)$-McCoy}, provided it is both left and right $(S,\omega)$-McCoy.
\end{definition}


Before stating and proving our chief result, we need a series of preliminary technicalities.

\begin{lemma}\label{un} Let $R$ be an abelian semi-regular ring with $J(R)$ nilpotent. If $\sum_{i=0}^nRa_iR=R$, then $\sum_{i=1}^nRa_i=R$.
\end{lemma}

\begin{proof} Assume that $r_0a_0s_0+\cdots+r_na_ns_n=1$. We may assume also that $\{a_{t_0},\cdots,a_{t_k}\}$ is a subset of $\{a_0,\cdots,a_n\}$ maximal with respect to $a_{t_i}\not\in J(R)$, $0\leq i\leq k$. Then, we may write $$r_{t_0}a_{t_0}s_{t_0}+\cdots+r_{t_k}a_{t_k}s_{t_k}+j=1,$$ where $j\in J(R)$. As $R$ is a semi-regular ring, there exist idempotents $e_i$ and elements $J_i\in J(R)$ such that $a_{t_i}=e_i+j_i$ for every $0\leq i\leq k$. It follows now that $$r_{t_0}e_0s_{t_0}+\cdots+r_{t_k}e_ks_{t_k}+r_{t_0}j_0s_{t_0}+\cdots+r_{t_k}j_ks_{t_k}+j=1.$$ However, since $R$ is abelian, we obtain $$r_{t_0}s_{t_0}e_0+\cdots +r_{t_k}s_{t_k}e_k+r_{t_0}j_0s_{t_0}+\cdots+r_{t_k}j_ks_{t_k}+j=1.$$ As $e_i=a_{t_i}-j_i$, $0\leq i\leq k$, we can get  $$r_{t_0}s_{t_0}a_{t_0}+\cdots +r_{t_k}s_{t_k}a_{t_k}-r_{t_0}s_{t_0}j_0-\cdots -r_{t_k}s_{t_k}j_k+(r_{t_0}j_0s_{t_0}+\cdots+r_{t_k}j_ks_{t_k}+j)=1.$$ Thus, $$r_{t_0}s_{t_0}a_{t_0}+\cdots +r_{t_k}s_{t_k}a_{t_k}=1+j^{\prime},$$ where $$j^{\prime}=r_{t_0}s_{t_0}j_0+\cdots +r_{t_k}s_{t_k}j_k-(r_{t_0}j_0s_{t_0}+\cdots+r_{t_k}j_ks_{t_k}+j).$$ But $J(R)$ is nilpotent, whence there exists an integer $m$ such that $(j^{\prime})^m=0$. Therefore,  $$(1-j^{\prime}+(j^{\prime})^2+\cdots+(-1)^{m-1}(j^{\prime})^{m-1})(r_{t_0}s_{t_0}a_{t_0}+\cdots +r_{t_k}s_{t_k}a_{t_k})=$$
$$(1-j^{\prime}+(j^{\prime})^2+\cdots+(-1)^{m-1}(j^{\prime})^{m-1})(1+j^{\prime}).$$ It follows now that $$(1-j^{\prime}+(j^{\prime})^2+\cdots+(-1)^{m-1}(j^{\prime})^{m-1})(r_{t_0}s_{t_0}a_{t_0}+\cdots +r_{t_k}s_{t_k}a_{t_k})=1.$$ Consequently, $Ra_{t_0}+\cdots+Ra_{t_k}=R$ and, since $$Ra_{t_0}+\cdots+Ra_{t_k}\subseteq Ra_0+\cdots+Ra_n,$$ we finally extract that $$Ra_0+\cdots+Ra_n=R,$$ as promised.
\end{proof}

Recall that a ring is {\it quasi-duo} if every maximal one-sided ideal is two-sided.

\begin{lemma}\label{2} Abelian semi-regular rings are quasi-duo.
\end{lemma}

\begin{proof} Assume that $M$ is a maximal right ideal of $R$. Thus, it must be that $J(R)\subset M$ and $$M=\sum_{m\in M}mR=\sum_{m\in M,\ m\not\in J(R)}mR+\sum_{n\in J(R)}nR=$$
$$\sum_{m\in M,\ m\not\in J(R)}mR+J(R).$$ As $R$ is a simi-regular ring, there exists idempotents $e_m\in R$ and $j_m\in J(R)$ such that $m=e_m+j_m$ for every $m\in M\setminus J(R)$. It, thus, follows that $$M=\sum_{m\in M,\  m\not\in J(R)}e_mR+J(R).$$ As $R$ is abelian, we write $$\sum_{m\in M,\ m\not\in J(R)}e_mR=\sum_{m\in M,\ m\not\in J(R)}Re_m=\sum_{m\in M,\ m\not\in J(R)}Re_mR,$$ and so $$M=\sum_{m\in M,\ m\not\in J(R)}Re_mR+J(R),$$ manifestly implying that $M$ is an ideal of $R$. In a pretty similar way, we can see that any maximal left ideal of $R$ is an ideal of $R$, as required.
\end{proof}

The next claim is an useful observation.

\begin{lemma}\label{3} Every maximal left (resp., right) ideal of a quasi-duo ring is a prime ideal.
\end{lemma}

\begin{proof} It follows from the facts that any maximal left (resp., right) ideal of a quasi-duo ring is two-sided ideal, and any maximal ideal is a prime ideal.
\end{proof}

A direct consequence is the following one.

\begin{corollary}\label{4} Every maximal left (resp., right) ideal of an abelian semi-regular ring is a prime ideal.
\end{corollary}

The following three assertions are pivotal for our major considerations.

\begin{lemma}\label{5} Let $R$ be an abelian semi-regular ring with $J(R)$ nilpotent. Then, $r_R(M)\neq 0$ for every maximal left (resp., right) ideal of $R$.
\end{lemma}

\begin{proof} Assume that $M$ is a maximal left ideal of $R$. Thus, Corollary \ref{4} tells us that $M$ is a prime ideal of $R$. Let $Q$ be the intersection of all maximal left ideals $P$ of $R$ such that $P\neq M$. Hence, $J(R)=M\cap Q$. As $Q$ is an ideal of $R$, we derive the inclusion $MQ\subseteq (M\cap Q)=J(R)$. Letting $q\in Q$ such that $q\not\in J(R)$, we deduce $Mq\subseteq J(R)$. But, since $R$ is semi-regular and $q\not\in J(R)$, there exist a non-zero idempotent $e\in R$ and an element $j\in J(R)$ such that $q=e+j$. It follows now that  $$M(q-j)=Me=eMe\subseteq J(R).$$ As $J(R)$ is nilpotent, there exists an integer $k$ such that $(Me)^k=\{0\}\neq (Me)^{k-1}$. However, $R$ being abelian gives that $M^ke=\{0\}\neq M^{k-1}e$. Consequently, $\{0\}\neq M^{k-1}e\leq r_R(M)$, as wanted.
\end{proof}

\begin{lemma}\label{7} Let $R$ be a ring, $(S,\preceq)$ a strictly totally ordered monoid, and $\omega: S\rightarrow {\rm End}(R)$ a monoid homomorphism. Assume that $R$ is $S$-compatible abelian, semi-regular with $J(R)$ nilpotent.  Then, $ab\in M$ if, and only if, $a\omega_s(b)\in M$ for every $s\in S$ and every maximal right (resp., left) ideal $M$ of $R$.
\end{lemma}

\begin{proof}
Assume that $ab\in M$, where $M$ is a maximal right ideal of $R$. Let $K$ be the intersection of all maximal right ideals of $R$ distinct of $M$. As $R$ is a quasi-duo ring, $K$ is an ideal of $R$ and so $MK\leq M\cap K=J(R)$. Hence, $abK\subseteq J(R)$. So, for every $k\in K$ with $k\not\in M$, we infer $abRkR\subseteq J(R)$. Since $J(R)$ is nilpotent, there exists an integer $n$ such that $(abRkR)^n=\{0\}$. As $R$ is $S$-compatible, we can write $(a\omega_s(b)RkR)^n=\{0\}$ for every $s\in S$. It, therefore, follows that $$a\omega_s(b)RkR\subseteq Nil_{*}(R)\subseteq J(R)\subseteq M.$$ Now, applying Lemma \ref{3}, and bearing in mind that $k\not\in M$, we conclude $a\omega_s(b)\in M$.\\ By similar arguments, we can see that if, for every $s\in S$, $a\omega_s(b)\in M$, then $ab\in M$, as desired.
\end{proof}

\begin{lemma}\label{8} Let $R$ be a ring, $(S,\preceq)$ a strictly totally ordered monoid and $\omega: S\rightarrow {\rm End}(R)$ a monoid homomorphism. Assume that $R$ is $S$-compatible abelian, semi-regular with $J(R)$ nilpotent. Assume also that $f,g\in A$ with $fg=0$, where $g\neq 0$. If, for every maximal right ideal $M$ of $R$, $C_{f}\nsubseteq M$, then $C_g \subseteq J(R)$.
\end{lemma}

\begin{proof} Assume, on the contrary, that there exists a maximal ideal $P$ of $R$ such that $C_g\nsubseteq P$. Set $T:=\{t\in Supp(g)\mid \ g(t)\not\in P\}$. As $S$ is a strictly totally ordered monoid, we can write $$T=\{t_0,t_1,\cdots\},\ \text{where}\ t_i\prec t_{i+1}\ \text{for}\ i=0,1,\ldots.$$ Since $C_f\not\in M$ for every maximal ideal $M$ of $R$, we receive $C_f\nsubseteq P$. Besides, put $W:=\{s\in Supp(f)\mid \ f(s)\not\in P\}$. As $S$ is a strictly totally ordered monoid, we can write $$W=\{s_0,s_1,\cdots\},\ \text{where}\ s_i\prec s_{i+1}\ \text{for}\ i=0,1,\ldots.$$
Note that $f(s)\in P$ and $g(t)\in P$ for every $s,t \in S$ with $s \prec s_0$ and $t\prec t_0$.

Consider now the set $X_{s_0t_0}(f,g)$. If $(s_0,t_0)\neq (s,t)\in X_{s_0t_0}(f,g)$, then we can plainly see that $s\neq s_0$ and $t\neq t_0$. Hence, we get either $s\prec s_0$, $t\succ t_0$ or $s_0\prec s$, $t_0\succ t$.\\
Notice that, for every $s\prec s_0$, we may choose $s\in Supp(f)\backslash W$, so by what we have shown above $f(s)\in P$, and so $f(s)g(t)\in P$ for all $t\in S$. Similarly, if $t\prec t_0$, then $g(t)\in P$ and thus $f(s)g(t)\in P$ for all $s\in S$. Employing Lemma \ref{7}, we detect that $$ \sum_{(s,t)\in X_{s_0t_0},\ s\prec s_0}f(s)\omega_s(g(t))\in P \ \text{and}\ \sum_{(s,t)\in X_{s_0t_0},\ t\prec t_0}f(s)\omega_s(g(t))\in P.\indent (*)$$
So,
$$0=fg(s_0t_0)=f(s_0)\omega_{s_0}(g(t_0))+\sum_{(s,t)\in X_{s_0t_0},\ s\prec s_0}f(s)\omega_s(g(t))+\\$$ $$\sum_{(s,t)\in X_{s_0t_0},\ t\prec t_0}f(s)\omega_s(g(t)).$$ From this and $(*)$, we arrive at  $f(s_0)\omega_{s_0}(g(t_0))\in P$, and usage of Lemma \ref{7} ensures that $f(s_0)g(t_0)\in P$.
As $R$ is a semi-regular ring, there exist  elements $r_0\in R$ and $j_0\in J(R)$ for which $f(s_0)r_0f(s_0)+j_0=f(s_0)$. But, as idempotents lift module $J(R)$ and $r_0f(s_0)+J(R)$ is an idempotent of $R/J(R)$, there exists a central idempotent $e_0=e_0^2$ of $R$ such that $r_0f(s_0)+J(R)=e_0+J(R)$. Hence, there exists $j_1\in J(R)$ such that $r_0f(s_0)+j_1=e_0$ and so $$f(s_0)r_0f(s_0)+f(s_0)j_1=f(s_0)e_0=e_0f(s_0).$$
Furthermore, $f(s_0)g(t_0)\in P$ assures that $r_0f(s_0)g(t_0)\in P$. As $J(R)\subseteq P$, we find that $j_1g(t_0)\in P$, and so $(r_0f(s_0)+j_1)g(t_0)\in P$. Now, it follows at once that $e_0g(t_0)\in P$. However, since $R$ is abelian, it must be that $Re_0=e_0R$, and thus $e_0Rg(t_0)\subseteq P$.

It, thereby, follows that $(r_0f(s_0)+j_1)Rg(t_0)\subseteq P$. As $j_1Rg(t_0)\subseteq P$, we must have $r_0f(s_0)Rg(t_0)\subseteq P$. Multiplying the equation on the left side by $f(s_0)$, we get $f(s_0)r_0f(s_0)Rg(t_0)\subseteq P$. It, thus, follows that $$f(s_0)Rg(t_0)=(f(s_0)r_0f(s_0)+j_0)Rg(t_0)\subseteq P.$$ Consequently, $f(s_0)Rg(t_0)\subseteq P$ and, since $P$ is a prime ideal, we conclude either $f(s_0)\in P$ or $g(t_0)\in P$ leading to an obvious contradiction. Finally, $C_g\subseteq J(R)$, as needed.
\end{proof}

We now have all the machinery necessary to establish the following main statement.

\begin{theorem}\label{main} Let $R$ be an $S$-compatible ring, $(S,\preceq)$ a strictly totally ordered monoid, and $\omega: S\rightarrow {\rm End}(R)$ a monoid homomorphism. If $R$ is abelian and semi-regular with $J(R)$ nilpotent, then $R$ is $(S,\omega)$-McCoy.
\end{theorem}

\begin{proof} We prove only the right case, because the proof of the left case is quite similar. \\

To this target, set $A:=R[[S,\omega,\preceq]]$ and also choose $f,g\in A$ with $fg=0$, where $g\neq 0$. Put also $C_{f} := \{f(t) :\  t \in Supp(f)\}$. Thus, the proof will be complete if we succeed to show that there exists a maximal right ideal $M$ of $R$ such that $C_f\subseteq M$, because, if this is the case, then invoking Lemma \ref{5}, we shall deduce that $r_R(C_f)\neq 0$ and so $f(s).\omega_s(r)=0$ for each $s\in Supp(f)$, where $0\neq r\in r_R(C_f)$.\\\par

Assume, on the contrary, that $C_f\nsubseteq M$ for each maximal right ideal $M$ of $R$. Then, appealing to Lemma \ref{8}, $C_g\subseteq J(R)$. So, since $J(R)$ is nilpotent and $C_g\subseteq J(R)$, there exists an integer $k$ minimal with respect to the condition that $$J(R)^{k+1} g(s)=\{0\}\ \text{ for every}\ s\in S. \indent (*_1)$$ \\ As $C_f\nsubseteq M$ for every maximal ideal $M$ of $R$, we claim that $RC_fR=R$ as for otherwise, in conjunction the fact that any proper ideal is contained in at least a maximal right ideal, we get a contradiction. So, $RC_fR=R$ indeed, and thus there exists a finite subset $\{s_0,s_1,\cdots ,s_m\}\subseteq Supp(f)$, where $s_i\prec s_{i+1}\ \text{for}\ i=0,1,\ldots,m$ and $f(s)=0$ for every $s\prec  s_0$. Therefore, $$\{ s_0,s_1,\cdots ,s_m\}=\{s\in Supp(f)\mid s\prec s_m\}\bigcup \{s_m\}\indent (*_2)$$ such that $ r_0f(s_0)l_0+\cdots+r_mf(s_m)l_m=1$, where $r_i,l_i\in R$, $0\leq i\leq m$. That is why, we infer $\sum_{i=0}^mRf(s_i)R=R$. However, Lemma \ref{un} works to see that $\sum_{i=0}^mRf(s_i)=R$, and so we can get $$r_R(\{f(s_0),\cdots,f(s_m)\})=0.\indent \indent\indent \indent\indent \indent (*_3)$$ From this and $(*_1)$, we obtain that there exists $t_0\in Supp(g)$ such that $$J(R)^kg(t_0)\neq \{0\}=J(R)^kg(t^{\prime}) \ \text{for every}\ t^{\prime}\prec t_0.\indent (*_4)$$

We now differ two basic cases:

\medskip

\textbf{Claim 1:} We assert that $J(R)^kf(s_0)g(t)=f(s_0)J(R)^kg(t)=\{0\}$ for each $t\in Supp(g)$ proving it by  induction. In fact, as $fg(s_0t_0)=0$ in  $R$, then $$\{0\}=J(R)^{k}(fg(s_0t_0) )=J(R)^kf(s_0)\omega_{s_0}(g(t_0))+$$$$J(R)^k\sum_{(s,t)\in X_{s_0t_0},\ s\prec s_0}f(s)\omega_s(g(t))+J(R)^k\sum_{(s,t)\in X_{s_0t_0},\ s_0\prec s}f(s)\omega_s(g(t)).$$ Moreover, note that, if $(s,t)\in X_{s_0t_0}$ with $s_0\prec s$, then we must have $t\prec t_0$, because if $t_0\prec t$ it must be that $s_0t_0\prec st$ -- a contradiction. Analogously, if $(s,t)\in X_{s_0t_0}$ with $t\prec t_0$, then we must have $s_0\prec s$. Thus, we receive
$$\sum_{(s,t)\in X_{s_0t_0},\ s_0\prec s}f(s)\omega_s(g(t))=\sum_{(s,t)\in X_{s_0t_0},\ t\prec t_0}f(s)\omega_s(g(t))=\{0\}$$ and, therefore, $$\{0\}=J(R)^{k}(fg(s_0t_0) )=J(R)^kf(s_0)\omega_{s_0}(g(t_0))+$$$$J(R)^k\sum_{(s,t)\in X_{s_0t_0},\ s\prec s_0}f(s)\omega_s(g(t))+J(R)^k\sum_{(s,t)\in X_{s_0t_0},\ t\prec t_0}f(s)\omega_s(g(t))$$
$$=J(R)^kf(s_0)\omega_{s_0}(g(t_0)).$$ But, since by assumption $R$ is $S$-compatible, we conclude $$J(R)^kf(s_0)g(t_0)=\{0\}.$$
Next, we manage to show that $f(s_0)J(R)^kg(t_0)=\{0\}$. To that purpose, if $f(s_0)\in J(R)$, then under presence of $(*_1)$ we arrive at $f(s_0)J^k(R)g(t_0)=\{0\}$. If $f(s_0)\not\in J(R)$, then since $R$ is semi-regular ring, there exists $v_0\in R$ such that $f(s_0)+J(R)=f(s_0)v_0f(s_0)+J(R)$. As all idempotents lift modulo $J(R)$ and $v_0f(s_0)+J(R)$ is an idempotent of the factor-ring $R/J(R)$, there exists a central idempotent $f_0=f_0^2$ of $R$ such that $v_0f(s_0)+u_0=f_0$, where $u_0\in J(R)$.\\

Furthermore, as $J(R)^kv_0\subseteq J^{k}(R)$ and $J(R)^kf(s_0)g(t_0)=\{0\}$, we obtain $J(R)^kv_0f(s_0)g(t_0)=\{0\}$. It now follows that $J(R)^k(f_0-u_0)g(t_0)=\{0\}$. Also, since $u_0\in J(R)$, utilizing $(*_1)$ we have $J(R)^ku_0g(t_0)=\{0\}$. Consequently, we infer that $J(R)^kf_0g(t_0)=\{0\}$. Thus, as $f_0$ is a central idempotent of $R$, we get $f_0J(R)^kg(t_0)\{0\}$. So, $(v_0f(s_0)+u_0)J(R)^kg(t_0)=\{0\}$ and since, by $(*_1)$, $u_0J(R)^kg(t_0)=\{0\}$, we detect that $v_0f(s_0)J(R)^kg(t_0)=0$. Multiplying the equality on the left side by $f(s_0)$, we can write $f(s_0)v_0f(s_0)J(R)^kg(t_0)=\{0\}$. From this and $(*_1)$, we have  $(f(s_0)v_0f(s_0)+J(R))J(R)^kg(t_0)=\{0\}$. However, since $f(s_0)v_0f(s_0)+J(R)=f(s_0)+J(R)$, we get $(f(s_0)+J(R))J(R)^kg(t_0)=\{0\}$. But, owing to $(*_1)$, we must have $$f(s_0)J(R)^kg(t_0)=\{0\}.\indent\indent \indent \indent\indent \indent \indent (*_5)$$
Hence, the base of the induction hypothesis is true.

Now, assume that $$J(R)^kf(s_0)g(t_i)=\{0\}=f(s_0)J(R)^kg(t_i)=\{0\}$$ for $0\leq i\leq n$, where $Supp(g)=\{t_0,t_1,t_2,\cdots\}$ with $t_0\prec t_1\prec \cdots$. As $fg(s_0t_{n+1})=0$ in $R$, then $$\{0\}=f(s_0)J(R)^{k}(fg(s_0t_{n+1}) )=f(s_0)J(R)^kf(s_0)\omega_{s_0}(g(t_{n+1}))+$$$$f(s_0)J(R)^k\sum_{(s,t)\in X_{s_0t_{n+1}},\ s\prec s_0}f(s)\omega_s(g(t))+$$$$f(s_0)J(R)^k\sum_{(s,t)\in X_{s_0t_{n+1}},\ t\prec t_{n+1}}f(s)\omega_s(g(t)).$$ And since $R$ is $S$-compatible and $$J(R)^kf(s_0)g(t_i)=\{0\}=f(s_0)J(R)^kg(t_i)$$ for $0\leq i\leq n$, it follows from $J(R)^kf(s)\subseteq J(R)^k$ that $$f(s_0)J(R)^k\sum_{(s,t)\in X_{s_0t_{n+1}},\ t\prec t_{n+1}}f(s)\omega_s(g(t))=\{0\},$$ and from $f(s)=0$ for $s\prec s_0$, we have $$f(s_0)J(R)^k\sum_{(s,t)\in X_{s_0t_{n+1}},\ s\prec s_0}f(s)\omega_s(g(t))=\{0\}.$$ It, therefore, follows that $$f(s_0)J(R)^kf(s_0)\omega_{s_0}(g(t_{n+1}))=\{0\}.$$ But $R$ being $S$-compatible insures that $$f(s_0)J(R)^kf(s_0) g(t_{n+1})=\{0\}.$$

Next, we would like to show that $f(s_0)J(R)^kg(t_{n+1})=\{0\}$. To this aim, if $f(s_0)\in J(R)$, then by $(*_1)$ we derive $f(s_0)J^k(R)g(t_{n+1})=\{0\}$. If $f(s_0)\not\in J(R)$, then since $R$ is a semi-regular ring, there exists $v_0\in R$ such that $f(s_0)+J(R)=f(s_0)v_0f(s_0)+J(R)$. As all idempotents lift modulo $J(R)$ and $v_0f(s_0)+J(R)$ is an idempotent of the quotient-ring $R/J(R)$, there exists a central idempotent $f_0=f_0^2$ of $R$ such that $v_0f(s_0)+u_0=f_0$, where $u_0\in J(R)$.\\
Likewise, as $J(R)^kv_0\subseteq J^{k}(R)$, we write $v_0f(s_0)J(R)^kv_0f(s_0)g(t_{n+1})=\{0\}$. It, thereby, follows that $(f_0-u_0)J(R)^k(f_0-u_0)g(t_{n+1})=\{0\}$. As $u_0\in J(R)$, having in mind $(*_1)$, we get $(f_0-u_0)J(R)^ku_0g(t_{n+1})=\{0\}$. Consequently, we obtain that $(f_0-u_0)J(R)^kf_0g(t_{n+1})=\{0\}$. As $u_0\in J(R)$, taking into account $(*_1)$, we get $u_0J(R)^kf_0g(t_{n+1})=\{0\}$. So, $f_0J(R)^kf_0g(t_{n+1})=\{0\}$. As $f_0$ is a central idempotent of $R$, we have $f_0J(R)^kg(t_{n+1})=\{0\}$. Hence, $(v_0f(s_0)+u_0)J(R)^kg(t_{n+1})=\{0\}$ and, since by $(*_1)$, $u_0J(R)^kg(t_{n+1})=\{0\}$, we observe that $v_0f(s_0)J(R)^kg(t_{n+1})=\{0\}$. Multiplying the equality on the left side by $f(s_0)$, we get $f(s_0)v_0f(s_0)J(R)^kg(t_{n+1})=\{0\}$. From this and $(*_1)$, we have $(f(s_0)v_0f(s_0)+J(R))J(R)^kg(t_{n+1})=\{0\}$. As $f(s_0)v_0f(s_0)+J(R)=f(s_0)+J(R)$, we deduce $(f(s_0)+J(R))J(R)^kg(t_{n+1})=\{0\}$. An appeal to $(*_1)$ guarantees that $$f(s_0)J(R)^kg(t_{n+1})=\{0\}.\indent\indent \indent \indent\indent \indent \indent $$

Similarly, we can show that $$J(R)^kf(s_0)g(t_{n+1})=\{0\}.\indent\indent$$ Thus, the induction hypothesis is really true, and we must have $$f(s_0)J(R)^kg(t)=J(R)^kf(s_0)g(t)=\{0\} \text{for every}\ t\in Supp(g).\indent\indent (*_6)$$

\medskip

\textbf{Claim 2:} We assert that $J(R)^kf(s_1)g(t)=f(s_1)J(R)^kg(t)=\{0\}$ for every $t\in Supp(g)$. For proving that, we use an induction. In fact, as $fg(s_1t_0)=0$ in $R$, then $$\{0\}=J(R)^{k}(fg(s_1t_0))=J(R)^kf(s_1)\omega_{s_1}(g(t_0))+$$$$J(R)^k\sum_{(s,t)\in X_{s_1t_0},\ s\prec s_1}f(s)\omega_s(g(t))+J(R)^k\sum_{(s,t)\in X_{s_1t_0},\ t\prec t_0}f(s)\omega_s(g(t)).$$ In virtue of the $S$-compatibility of $R$, and in view of $(*_4)$ and $(*_6)$, we can get $$J(R)^k\sum_{(s,t)\in X_{s_0t_0},\ s\prec s_1}f(s)\omega_s(g(t))=J(R)^k\sum_{(s,t)\in X_{s_0t_0},\ t\prec t_0}f(s)\omega_s(g(t))=\{0\}.$$ So, one follows that $$J(R)^{k}(fg(s_1t_0) )=J(R)^kf(s_1)\omega_{s_1}(g(t_0))=\{0\}.$$ Since $R$ is $S$-compatible, we have $$J(R)^kf(s_1)g(t_0)=\{0\}.$$
Arguing as above, we obtain $$f(s_1)J(R)^kg(t_0)=\{0\}.$$ Thus, the base of induction is true.\\

Now, assume that $$J(R)^kf(s_1)g(t_i)=\{0\}=f(s_1)J(R)^kg(t_i)$$ for $0\leq i\leq n$, where $t_0\prec t_1\prec \cdots$. But, as $fg(s_1t_{n+1})=0$ in $R$, we infer that $$\{0\}=J(R)^{k}f(s_1)(fg(s_1t_{n+1}) )=J(R)^kf(s_1)^2\omega_{s_1}(g(t_{n+1}))+$$$$J(R)^kf(s_1)\sum_{(s,t)\in X_{s_1t_{n+1}},\ s\prec s_1}f(s)\omega_s(g(t))+J(R)^kf(s_1)\sum_{(s,t)\in X_{s_1t_{n+1}},\ t\prec t_{n+1}}f(s)\omega_s(g(t)).$$
Thus, both relations $J(R)^kf(s_1)\subseteq J(R)^k$ and $f(s_0)J(R)^kg(t)=f(s_0)J(R)^kg(t)=\{0\}$ for every $t\in Supp(g)$ lead us to $$J(R)^kf(s_1)\sum_{(s,t)\in X_{s_1t_{n+1}},\ s\prec s_1}f(s)\omega_s(g(t))=\{0\}.$$ If, for a moment, $f(s)\in J(R)$, then by $(*_1)$ we get $J(R)^kf(s_1)f(s)\omega_s(g(t))=\{0\}$ for every $t\in Supp(g)$. If, however, $f(s)\not\in J(R)$, then there exists a central idempotent $f\in R$ such that $f(s)=f+j$. So,  $$J(R)^kf(s_1)f(s)\omega_s(g(t_i))=J(R)^kf(s_1)(f+j)\omega_s(g(t_i)),$$ where $0\leq i\leq n$. Knowing $(*_1)$, it must be that $J(R)^kf(s_1)j\omega_s(g(t_i))=\{0\}$, and thus $$J(R)^kf(s_1)f(s)\omega_s(g(t_i))=J(R)^kf(s_1)(f+j)\omega_s(g(t_i))=$$$$J(R)^kf(s_1)f\omega_s(g(t_i))=
fJ(R)^kf(s_1)\omega_s(g(t_i))=\{0\},$$ where $0\leq i\leq n$. So, $J(R)^kf(s_1)f(s)\omega_s(g(t_i))=\{0\}$ for $0\leq i\leq n$, whence $$J(R)^kf(s_1)\sum_{(s,t)\in X_{s_1t_{n+1}},\ t\prec t_{n+1}}f(s)\omega_s(g(t))=\{0\}.$$ It, in turn, gives that $J(R)^kf(s_1)^2\omega_{s_1}(g(t_{n+1}))=\{0\}$. However, $R$ being $S$-compatible enables us that  $$J(R)^kf(s_1) ^2g(t_{n+1})=\{0\}.$$ From this and $(*_1)$, we write $$J(R)^k(f(s_1)+J(R))(f(s_1)+J(R))g(t_{n+1})=\{0\}.$$ Since $R$ is semi-regular ring, there exists an idempotent $f_1\in R$ such that $f(s_1)+J(R)=f_1+J(R)$. It, thereby, readily follows that $$\{0\}=J(R)^k(f_1+J(R))(f_1+J(R))g(t_{n+1})=$$$$J(R)^k(f_1+J(R))g(t_{n+1})=J(R)^k(f_(s_1)+J(R))g(t_{n+1}).$$ Further, in accordance with $(*_1)$, we get $$J(R)^kf(s_1)g(t_{n+1})=\{0\}.$$
In a way of similarity with the above argument, we also get $$f(s_1) J(R)^kg(t_{n+1})=\{0\}.$$ Thus, one sees that $$J(R)^kf(s_1) g(t)=f(s_1) J(R)^kg(t)=\{0\}\ \text{for every}\ t\in Supp(g).$$

Repeating the same procedure finitely many times, one easily obtains that $$J(R)^kf(s_i) g(t)=f(s_i) J(R)^kg(t)=\{0\}\ \text{for every}\ t\in Supp(g)\ \text{and}\ 0\leq i\leq m.$$
It, therefore, follows for every $t\in Supp(g)$ that $$J(R)^kg(t)\subseteq r_R(\{f(s_0),\cdots,f(s_m)\}).$$ But, with $(*_3)$ at hand, we receive $J(R)^kg(t)=\{0\}$ for every $t\in Supp(g)$, which is an obvious contradiction to $(*_4)$. So, after all considerations, there exists a maximal ideal $M$ of $R$ such that $C_f\subseteq M$, as expected, thus finishing the proof.
\end{proof}

The following five corollaries are immediate consequences of the preceding theorem.\\

\begin{corollary} Let $R$ be an $S$-compatible ring, $(S,\preceq)$ a strictly totally ordered monoid, and $\omega: S\rightarrow {\rm End}(R)$ a monoid homomorphism. If $R$ is abelian and semi-regular with $J(R)$ nilpotent, then the skew Malcev-Neumann series ring $R((S,\omega))$ is skew McCoy.
\end{corollary}

\begin{corollary}
Let $R$ be a ring and $\alpha$ a compatible automorphism of $R$. If $R$ is abelian and semi-regular with $J(R)$ nilpotent, then the skew Laurent series ring $R[[x, x^{-1}; \alpha]]$ is skew McCoy.
\end{corollary}

\begin{corollary}
Let $R$ be an $S$-compatible ring, $(S,\preceq)$ a strictly totally ordered monoid. If $R$ is abelian and semi-regular with $J(R)$ nilpotent, then the generalized power series ring $R[[S]]$ is skew McCoy.
\end{corollary}

\begin{corollary} Let $R$ be an abelian and semi-regular ring with $J(R)$ nilpotent. Then, the Malcev-Neumann series ring $R((S))$ is McCoy.
\end{corollary}

\begin{corollary}
Let $R$ be an abelian and semi-regular ring with $J(R)$ nilpotent. Then, the Laurent series ring $R[[x, x^{-1}]]$ is  McCoy.
\end{corollary}

Let $R$ be a ring. Then, we define the matrix ring
$$S_n(R)=\left\{\left(
                                                   \begin{array}{cccc}
                                                     a_{11} & a_{12} & \cdots & a_{1n} \\
                                                     0 & a_{11} & \ddots & \vdots \\
                                                     \vdots & \vdots & \ddots & a_{12} \\
                                                     0 & 0 & \cdots & a_{11} \\
                                                   \end{array}
                                                 \right)|\ a_{ij}\in R,\ 1\leq i,j\leq n.
\right\}.$$

It is worthwhile noticing that every abelian semi-primary ring satisfies the conditions of our main result Theorem \ref{main}. We, thereby, come to the following affirmation.

\begin{example} The ring $S_n(R)$ is abelian semi-regular with nilpotent Jacobson radical (respectively, is abelian semi-primary) if, and only if, so is the ring $R$.
\end{example}

\begin{proof} Since one verifies that the isomorphism $R/J(R)\simeq S_n(R)/J(S_n(R))$ holds, we conclude that $R/J(R)$ is regular (semi-simple) and $J(R)$ is nilpotent if, and only if, $S_n(R)/J(S_n(R))$ is regular (semi-simple) and $J(S_n(R))$ is nilpotent. Note that the sum

\medskip

$$\left(
                                                   \begin{array}{cccc}
                                                     e_{11} & e_{12} & \cdots & e_{1n} \\
                                                     0 & a_{11} & \ddots & \vdots \\
                                                     \vdots & \vdots & \ddots & e_{12} \\
                                                     0 & 0 & \cdots & e_{11} \\
                                                   \end{array}
                                                 \right)+J(S_n(R))$$

                                                 \medskip
is an idempotent of $S_n(R)/J(S_n(R))$ if, and only if, the sum $e_{11}+J(R)$ is an idempotent of $R/J(R)$. So, all idempotents in $R$ can be lifted modulo $J(R)$ if, and only if, all idempotents in $S_n(R)$ can be lifted modulo $J(S_n(R))$. But since, apparently, $R$ is abelian if, and only if, so is $S_n(R)$, we establish the stated assertion.
\end{proof}

We close our work with the following challenging question.

\medskip

\noindent{\bf Problem.} Is it true that Theorem~\ref{main} remains absolutely true if we replace the condition on $R$ to be "abelian" to the more general restriction of being "2-primal"?

\medskip




\end{document}